\documentclass[11pt]{amsart}

\usepackage{amsmath}
\usepackage{amssymb}
\usepackage{amsfonts}
\usepackage{mathrsfs}
\usepackage{amsthm}
\numberwithin{equation}{section}

\def \IN{\mathbb N}
\def \ik{\mathcal K}
\def \IZ{\mathbb Z}

\def \IR{\mathbb R}

\def \fs{\mathfrak{S}}

\def \mfG{{\mathfrak G}}

\def \mfq{{\mathfrak q}}

\newcommand{\gre}{{\epsilon}}
\newcommand{\grz}{{\zeta}}

\newcommand{\grl}{{\lambda}}

\newcommand{\arw}{\rightarrow} %(Example: f: \RR \arw \RR)
\newcommand{\D}{\displaystyle}
\newcommand{\mb}{\mathbf}
\newcommand{\Hil}{\mathscr{H}}
\newcommand{\alg}{\mathscr{L}}

\newcommand{\art}{\ar@{-}}

\newtheorem{main}{Theorem}[section]
\newtheorem{Intro_thm}[main]{Theorem}
\newtheorem{step1}[main]{Lemma}
\newtheorem{sl_cor}[main]{Corollary}
\newtheorem{ad_cor}[main]{Corollary}

\newtheorem{compact}[main]{Lemma}
\newtheorem{cartan}[main]{Lemma}
\newtheorem{central}[main]{Theorem}

\newtheorem{planch}[main]{Theorem}

\newtheorem{mainexp}[main]{Theorem}
\newtheorem{dens}[main]{Definition}
\newtheorem{corr}[main]{Corollary}
\newtheorem{corrr}[main]{Corollary}

\newtheorem{arch1}[main]{Corollary}
\newtheorem{broadest}[main]{Theorem}

\begin{document}
\title{Effective multiple mixing in semidirect product actions}
\author{Ioannis Konstantoulas}
\begin{abstract}We prove effective decay of certain multiple correlation 
coefficients for measure preserving, mixing Weyl chamber actions of semidirect products 
of semisimple groups with $G$-vector spaces.  These estimates provide decay for some 
semisimple groups of higher rank.
\end{abstract}
\maketitle
\tableofcontents
\section{Introduction}\label{introd}
Ergodicity and various notions of mixing constitute the measure theoretic 
incarnation of statistical uniformity in measurable dynamics.  The 
relationships among those notions are not completely understood; it is not 
known, for example, if every mixing map $T$ on a measure space $X$ is mixing 
of all orders.  For $\IZ^d$-actions, $d\geq 2$, Ledrappier gave a counterexample 
in \cite{Led} later generalized by Schmidt (see 
\cite{Sch}) through the introduction of certain algebraic dynamical systems.  
General $\IR^d$-actions can also have a wide range of behaviors depending on 
the nature of the action and properties of the object acted upon.

In the context of actions of a highly noncommutative group (such as a 
semisimple group) the situation changes 
drastically.  Many topological and measurable distribution phenomena follow 
not from the peculiarities of the action, but rather on the nature of the 
acting group.  This is 
manifest in the following two basic results: the Howe - Moore vanishing 
theorem \cite{HM} which implies that every ergodic action of such a group is 
mixing and Mozes's theorem \cite{Moz} stating that for well behaved 
semisimple Lie groups, mixing implies multiple mixing of all orders.  It is 
the rich geometry and non trivial interplay between certain subgroups of 
semisimple groups that lies at the heart of both these results and accounts 
for this exceptional behavior of semisimple groups.  For example, Mautner's 
phenomenon (see \cite{Bekka}) in various guises plays a major role in both 
cited results.  A comprehensive survey emphasizing this interplay can be 
found in \cite{Starkov}.

 Mozes's theorem relies on geometric considerations concerning how the group 
acts on a certain space of measures.  It uses compactness in an essential 
way, preventing an immediate quantitative refinement even with specific 
hypotheses on the action and the target space.  The Howe - Moore theorem, on 
the other hand, is essentially a representation theoretic result; for this 
reason, it is possible to make quantitative by utilizing the $L^2$-theory and 
a clever application of tensor products.
 
 The next question that arises is whether we can give a quantitative form of 
Mozes's theorem whose asymptotics agree with Howe - Moore for $2$-mixing.  
More explicitly, let $G$ be a suitable group with a measure 
preserving action on a probability measure 
space $(X,\mu)$ denoted $g\cdot x$.  For any $k+1$-tuple $f_i \in L_0^\infty(X)$
(bounded, zero mean functions) form the correlation integral
\begin{equation*}
\int_X f_0(x)f_1(g^{-1}_1\cdot x)\cdots f_{k}(g^{-1}_k \cdot x)d\mu(x)
\end{equation*}

The task is to bound this integral in terms of data (as explicit as 
possible) intrinsic to the $f_i$, to the acting group $G$ and to a given 
notion of growth of the acting tuple $(g_i)$.  This paper is mainly concerned 
with the Weyl chamber action of groups $\mfG$ where $\mfG$ is the semidirect product of a semisimple group $G$ 
with a vector space by means of a representation of one on the other.  Such 
groups are naturally found as subgroups of semisimple groups (but are not 
restricted to such a role) and we exploit this inclusion to get quantitative 
decay for some semisimple groups of higher rank.  Since 
the precise statement of our result is too technical
to give in the introduction, we will give a simplified description here and refer
to Section 4 for the full version.  

Note that this paper relies heavily on 
the work of Wang in \cite{Wang} and cannot possibly be made self contained 
without copying that work verbatim.  References to proofs in \cite{Wang} will 
abound and the reader is advised to consult it for verifying several 
statements made here.

Following Wang we let $G$ be a connected semisimple almost algebraic (see
section \ref{defs} for precise definitions) $\ik$-group where $\ik$ is a local field 
of characteristic zero
together with a $\ik$-rational representation $\rho: G \arw \textrm{GL}(V)$ 
satisfying certain
conditions.  Let the group $G\ltimes_{\rho}V$ act on a probability space 
$(X,\mu)$ via measure-preserving transformations so that the action is 
mixing; thus we obtain a unitary representation
on $L^2(X)$ that distributes over pointwise (a.e.) products of functions.  Restricting
it to $G\ltimes_{\rho}\mb{0}$ we get a (still mixing) action of the semisimple group
$G$; relative to a Cartan decomposition $G=KD^+FK$ we let $k+1$ elements 
$$\mb{a}^i\in D^+$$ of $G$ with $\mb{a}^0 = I$ act on $k+1$ 
bounded, zero mean, $K$-finite functions $f_i$ of bounded spectral support
\footnote{See section \ref{defs} for definitions.}.  
Define 
$$\mathfrak{R}(\mb{a}) = \left|\sum_{i=1}^{k-1} 
\lambda(\frac{\mb{a}^k}{\mb{a}^i}) \right|\left|\sum_{i=1}^k 
\varrho(\mb{a}^i)^{-1} \right| $$
where $\varrho$ and $\lambda$ are lowest and highest weights respectively for the representation $\rho$.
Our main result (\ref{centralt}) will imply the following estimate:
\begin{Intro_thm}
Let $\mb{a}^i$ be as above; there exists an $L^2$-dense subspace $\mathcal{D}$
 of bounded, zero-mean, $K$-finite functions and constants $C,C'$ depending 
only on the action\footnote{By 'action' we will refer to all data involved in 
it, including the group structure.  The independence is on the functions in the correlation integral.} so that if 
$$\min(\min_{i=0,\cdots,k-1}|\lambda(\frac{\mb{a}^k}{\mb{a}^i})|,\min_{i=1,\cdots,k}|\varrho(\mb{a}^i)^{-1}|)>C',$$ $f_i$ are in $\mathcal{D}$ and 
$d_i=\textrm{dim}\langle K\cdot f_i\rangle$, we have the bound
\begin{equation}
 \left|\int_X f_0(x)f_1((\mb{a}^1)^{-1}\cdot x)\cdots f_{k}((\mb{a}^k)^{-1} \cdot x)d
\mu(x)\right| \leq s^{2\mathfrak{q}} C\prod_i ||f_i||_\infty d_i \mathfrak{R}(\mb{a})^{-\frac{\mathfrak{q}}{2}}
\end{equation}
where $s$ depends on the functions\footnote{Explicitly, $s$ is such that all the $f_i$ lie
in the image $P_s(\alg_{0,K})$ for the approximate projection $P_s$ defined 
in section \ref{defs}.} $f_i$ and the action; the exponent $\mathfrak{q}$ only depends on the action.
\end{Intro_thm}
 For the action on 
smooth vectors we will get bounds for a wider class of functions defined 
by the finiteness of certain Sobolev norm (i.e. a norm involving 
derivatives of various orders in certain directions in the group) for the 
functions $\sigma(g)\cdot f:G \arw L^\infty(X)$.

Note that if a semisimple group of higher rank has an adequate number 
of subgroups that are semidirect products of the form above (for example, 
covering the entire Cartan action), we can 
extend the bounds to those groups.  We will give the extension briefly for 
the group $\textrm{SL}(n,\ik)$ in Section \ref{SLN}, following the demonstration 
given in \cite{HT}, and discuss how to get such bounds in some generality in Section \ref{simple_groups}.

Our method was inspired by a proof of quantitative decay 
given in \cite{HT}.  Initially we had worked out the case $SL(n,\IR)$ only, 
when the paper of Zhenqi Jenny Wang \cite{Wang} was brought to our attention 
which generalized computations in 
\cite{HT}; this led us to expand the setting and try to isolate parts of the 
computation that can work in more general settings.  Apart from methods, we 
borrow many notations and notions from these expositions, so the reader is 
advised to consult them for further reference.  For complete proofs of 
unproven statements found here and in Howe-Tan's book, see \cite{Lang}, 
chapters 1, 2 and Appendices 1-3.

While this paper was being written, I was notified that similar results were 
obtained by Bjorklund, Einsiedler and Gorodnik in \cite{BEG}; that 
work treats the full action of semisimple groups over local fields and adeles 
acting on suitable homogeneous spaces.  It covers actions of rank $1$ groups with spectral gap as well as 
nonsplit groups, two situations that that we cannot treat here, and provides 
uniform estimates for Sobolev vectors.  Their method is dynamical in nature 
and examines the quantitative properties of the orbit of the correlation 
measures by the acting elements; in that sense, it is close to the spirit of 
Mozes' original method.  In contrast, our method is spectral and 
builds on consecutive approximations by nicely behaved functions.  Since semidirect products are our 
main focus, for the application of our estimates to semisimple groups we need these 
groups to contain sufficiently many semidirect products.  This excludes 
groups of split rank at most $1$ and almost direct products of such groups.

\textbf{Acknowledgements: }  I would like to thank my adviser Jayadev Athreya 
for introducing me to this problem and providing the crucial materials this 
work was based on.  Furthermore, I would like to thank Han Li for useful 
observations about this work.    Finally, I would like to thank professor Alexander 
Gorodnik for his reading of versions of this paper, making corrections and suggesting 
significant improvements.

\section{Set up and central notions}\label{defs}
In the following sections, we lay down notation, the central objects of study 
and the tools we will use in the proofs.

\subsection{Semidirect products and excellent representations}\label{defs_semidirect}

Let $\ik$ be a local field of characteristic zero, i.e. the real or complex 
numbers or a finite extension of a $p$-adic field.  Having fixed $\ik$, whenever
we talk about an algebraic group as a functor we will indicate it by a tilde, e.g.
$\tilde{G}$.  Then $G$ will denote the group of $\ik$-rational points of $\tilde{G}$.

Let $G$ be the group of $\ik$-rational points of a connected semisimple algebraic
\footnote{In the case $\ik=\IR$ one can take $G$ to be any semisimple Lie 
group and our results are true in that generality.  We will restrict to 
algebraic groups for simplicity of notation.  See \cite{Wang} for the 
modifications that need to be made to accommodate this case.} group 
$\tilde{G}$.  Let $\tilde{D}$ be a maximal $\ik$-split torus and $\tilde{B}$ 
a minimal parabolic containing $\tilde{D}$. Write $\mb{X}(\tilde{D})$ for the 
characters of $\tilde{D}$ defined over $\ik$.  The choice of a parabolic 
group $B$ determines an ordering of the characters; 
let $\mb{X}^+$ be the set of positive characters with respect to the given 
ordering.

As in \cite{Wang} let $$\ik^0 = \{x\in \IR| x\geq 0\}\quad\textrm{ and }\quad 
\overline{\ik} = \{x\in \IR| x\geq 1\}$$ when $\ik$ is Archimedean.  When 
$\ik$ is non-Archimedean, we fix a uniformizer $q$ with $|q|^{-1}$ the 
cardinality of the residue field of $\ik$.  Then correspondingly
$$\ik^0 = \{q^n| n\in \IZ\}\quad\textrm{ and }\quad\overline{\ik} = \{q^{-n}| n\in \IN\}.$$
Define subgroups $D^0$ and $D^+$ of $D$ by
$$D^0 = \{d\in D | \chi(d) \in \ik^0\textrm{ for each }\chi\in \mb{X}(\tilde{D})\},$$
$$D^+ = \{d\in D | \chi(d) \in \overline{\ik}\textrm{ for each }\chi\in \mb{X}^+\}.$$
We call $D^+$ the positive Weyl chamber in $D$ (relative to the prescribed data).

Next, we denote the centralizer of $\tilde{D}$ in $\tilde{G}$ by $\tilde{Z}$ and transfer
the ordering of $\mb{X}(\tilde{D})$ to $\mb{X}(\tilde{Z})$ by inclusion.  Let
$$Z_+ = \{z\in Z | |\chi(z)|\geq 1\textrm{ for each }\chi \in \mb{X}(Z)^+\}$$ and
$$Z_0 = \{z\in Z | |\chi(z)|= 1\textrm{ for each }\chi \in \mb{X}(Z)^+\}.$$  Note
that these are $\ik$-subgroups of the $\ik$-group $Z$.  We then have the 
following decomposition (see \cite{Bor} and the discussion in \cite{Wang}):
\begin{cartan}
There exists a good maximal compact subgroup $K$ of $G$ such that
\begin{enumerate}
\item $N_G(D)\subset KD$.
\item We have the decomposition $G=K(Z_+/Z_0)K$ such that for each $g\in G$, there exists a unique
element $z$ of $Z_+$ modulo $Z_0$ so that $g\in KzK$.
\item There exists a finite subset $F\subset C_G(D)$ so that $G = K(D^+F)K$ and for each $g\in G$
there exist unique $d\in D^+$ and $f\in F$ so that $g\in KdfK$.
\end{enumerate}
\end{cartan}

Recall that any semisimple group is the almost direct product of its almost 
simple factors, $G = \prod G_i$.  We will assume that no factor is compact, 
although this can be avoided at the expense of making the statements of the 
theorems more complicated.  We opt for simplicity.  

Now let $\rho: G\arw \textrm{GL}(V)$ be a representation on 
a finite dimensional $\ik$-vector space $V$ with the following properties:
\begin{enumerate}
\item $\rho$ is continuous when $\ik=\IR$ and $\ik$-rational in all other cases;
\item for each almost simple factor $G_i$, the only $\rho(G_i)$-fixed point in $V$ is $\mb{0}$.
\end{enumerate}
Such representations are called excellent in \cite{Wang}.  We will also 
assume once and for all that $\textrm{ker}(\rho)< Z(G)$, i.e. that the 
representation is not far from faithful.  By means of this 
representation we define the main acting object of this work: let
\begin{equation}
\mathfrak{G} = G \ltimes_\rho V
\end{equation}
be the semidirect product of $G$ with $V$ by means of $\rho$.  This is a $\ik$
-group whose unipotent radical over $\ik$ coincides with $V$.  Since $G$ is 
semisimple and $\textrm{char}(\ik)=0$, the representation $\rho$ is 
completely reducible and thus $V$ breaks into irreducible components $$V = 
\oplus_{i=1}^N V_i.$$  Now we introduce notation involving $V$ and $\rho$:
\begin{enumerate}
\item $||\cdot||$ denotes a $K$-invariant norm on $V$.
\item The restriction of $\rho$ on $V_i$ is denoted by $\rho_i$.
\item $\Phi_i$ is the set of weights of $\rho_i$ with respect to $D$ on $V_i$.
\item $\lambda_i$ resp. $\varrho_i$ are the highest resp. lowest weights of $\rho_i$.
\item For each $i$ and each weight $w$ of $\rho_i$, $V_w$ is the 
corresponding weight subspace of $V_i$ (there will be no problem 
distinguishing irreducible components).
\item $\Phi$ is the set of roots of $G$ with respect to $D$.
\item For each $\omega \in\Phi$, denote by $\mathfrak{g}_\omega$ the root 
space corresponding to the root.
\item $\delta_B$ is the modular function of the Borel subgroup $B$ that 
determines the ordering on $\Phi$.
\item $\{\omega_1,\cdots,\omega_n\} \subset \Phi^+$ is the set of simple 
roots in $\Phi^+$.
\item $\mfq_i:= \left(\frac{1}{3}\right)^{\# \Phi_i-1}$ if 
$\textrm{dim}{V_{\grl_i}}>1$, otherwise $\mfq_i:= 
\left(\frac{1}{3}\right)^{\# \Phi_i-2}$.
%\item $\Lambda(\Phi_i) = \frac{\mfq_i}{2}(\grl_i-\grr_i)$.
%\item $\D{p(G,V_i,\Phi_i) = \max_{j=1,\cdots,n}{\frac{\textrm{coefficient of 
%}\omega_j\textrm{ in }\delta_B}{\textrm{coefficient of }\omega_j\textrm{ in 
%}\Lambda(\Phi_i)}}}$
%\item $p(G,V,\Phi) = \max_{i}(p(G,V_i,\Phi_i))$
\end{enumerate}
More details about the aspects of root systems and weights 
we will use can be found in \cite[Section 3]{Wang} and the references 
therein.  The concepts listed above will play a role in the explicit bounds 
we will give for the correlations that we will now introduce.

\subsection{Actions and unitary representations}\label{defs_unitary}

Let $(X,\mu)$ be a probability space, $\Hil=L_0^2(X)$ the Hilbert space of 
square integrable functions on $X$ orthogonal to the constants, 
$\langle\,\rangle$ the inner product, $\alg = L_0^\infty(X)\subset \Hil$ and 
$\sigma$ a measure-preserving, mixing action of $\mfG$ on $X$; we always use the 
notation $g\cdot x$ for $\sigma(g)(x)$.  The action on $\Hil$ defined by 
$$(g\cdot f)(x):=f(g^{-1}\cdot x) $$ is a unitary representation of $G$; we 
call $g\cdot f$ a translate of $f$ by $g$, suppressing mention of the action.  
Note that the representation is multiplicative, i.e. it distributes over 
pointwise (and a.e. pointwise) products of functions: $$g\cdot(fh) = (g\cdot 
f)(g\cdot h).$$  Furthermore, for each irreducible component $V_i$ of $V$, 
the representation $\sigma|_{V_i}$ has no fixed vectors in $\Hil$.  A fixed 
vector would give an invariant matrix coefficient on $\Hil$, but mixing 
implies the decay of all such coefficient.

For any linear space of functions on $X$, a subscript $K$ denotes $K$-finite 
functions in that space.  A function $f$ is called $K$-finite if the space 
$$\langle K\cdot f\rangle\subset \Hil$$ is finite dimensional.   We will deal especially 
with the algebra of $K$-finite bounded functions $\alg_K$ and $L^2$-dense 
subspaces in it.  As the representation induced on $L^2(X)$ by the action is 
unitary, by the Peter - Weyl theorem $K$-finite vectors are dense in $L^2(X)$.

\subsection{Form of the bounds}\label{defs_form}
After this setup, we can now describe the kind of correlations we will bound 
and in what way.  Recall the notation in the previous section.  Let $f_i$ be 
bounded, $K$-finite functions on $X$.  For simplicity, assume that $V=V_1$ so 
$\rho$ is irreducible.  Abbreviate $\varrho_1$ and $\lambda_1$ by $\varrho$ 
and $\lambda$.  Of course, in the general case we can restrict to 
each irreducible component, obtain bounds there and choose the best.  The 
mixing property provides us this luxury\footnote{So we see that all we need 
in fact is that there are no invariant vectors in $\Hil$ for any $V_i$.}. 

Let $\mb{a}^i\in D^+$ be the acting elements; for uniformity, define 
$\mb{a}^0 = \mb{\textrm{I}}$.  Order the elements according to their $\lambda$
 values, i.e. $|\lambda(\mb{a}^i)|>|\lambda(\mb{a}^j)|$ for $i>j$.

\begin{mainexp}\label{desc}
With notations as above, there exists an $L^2$-dense subspace $\mathcal{D}$ of bounded, zero mean, $K$-finite vectors such that, for \begin{equation} \min(\min_{i=0,\cdots,k-1}|\lambda(\frac{\mb{a}^k}{\mb{a}^i})|,\min_{i=1,\cdots,k}|\varrho(\mb{a}^i)^{-1}|)>C',\end{equation}\label{norms} $C'$ independent of the $f_i$,
\begin{eqnarray}
& & \int_X f_0(x)\,\mb{a}^1\cdot f_1( x)\,\cdots \,\mb{a}^k \cdot f_{k}(x)\,d\mu(x) \nonumber\\
& & \quad \leq s^{2\mfq} C \prod_i ||f_i||_{\infty} \textrm{dim}\langle K\cdot f_i\rangle \left|\sum_{i=1}^{k-1} \lambda(\frac{\mb{a}^k}{\mb{a}^i}) \right|^{-\frac{\mfq}{2}}\left|\sum_{i=1}^k \varrho(\mb{a}^i)^{-1} \right|^{-\frac{\mfq}{2}}
\end{eqnarray}
where $s$ is a parameter depending only on the $f_i$ and $C$ depends only on $G$ and the action, but not the $f_i$.
\end{mainexp}
In fact, this is weaker than what we will prove in Section 4, but this version is given here for simplicity.

Notice that the $\mb{a}^i$ are in a positive Weyl chamber (in $D$) and a 
given weight is evaluated on them in both terms on the right hand side, so 
the summands do not cancel each other.  The space $\mathcal{D}$ and the parameter $s$ will become explicit in the 
next sections; in later sections we will examine actions for which $s$ can 
be eliminated and $D$ is replaced by all bounded, zero mean $K$-finite 
vectors satisfying appropriate smoothness assumptions.  With this uniformity, 
we will then be able to pass to arbitrary Sobolev vectors (of sufficiently 
high order) by using an observation from \cite{KaSp} based on a computation 
in \cite[Lemma 4.4.2.3]{War}.

\subsection{Approximate Projections}\label{defs_approx}
Given a finite dimensional normed vector space $(V, ||\cdot ||)$ over $\ik$ 
with $K$-invariant norm, we denote by $\widehat{V}$ the unitary dual, i.e. 
the topological group of all additive unitary characters of $V$.  For 
$\mb{x}=(x_i),\, \mb{y}=(y_i) \in V$ let $$(\mb{x},\mb{y})=\sum{x_i y_i}$$ be 
the standard bilinear form on $V$.  Choosing a fixed non trivial unitary 
character $\zeta$ of $\ik$, define the map $V \arw \widehat{V}$ by $$v \arw 
\zeta((v,\cdot)) =: \zeta_v.$$  This correspondence is a topological group 
isomorphism between $V$ and $\widehat{V}$ through which we will usually 
identify the two.  In this situation, given $v,w\in V$, we denote $[v,w] = 
\zeta_v(w)$.

Under $(\cdot,\cdot)$ we naturally define the transpose of a linear operator; 
define $\rho^*:G\arw \textrm{GL}(V)$ to be the inverse transpose of $\rho$, 
$$\rho^*(v) := (\rho^{-1})^T(v).$$  This provides an identification of the 
dual action of $G$ on $V^*$ with the action $\rho^*$ on $V$, given the 
topological isomorphism above.  Furthermore, if $\rho$ is irreducible and 
excellent on $V$, so is $\rho^*$; finally, $\|\cdot\|$ is $\rho^*(K)$
-invariant as well.  See Section 6.1 of \cite{Wang} for these facts.

For $f\in L^1(V)$ and $\chi\in \widehat{V}$, define the Fourier transform
\begin{equation}
\widehat{f}(\chi) = \int_{V} \overline{\chi(v)} f(v)\, dm(v)
\end{equation}

Where $dm(v)$ is a Haar measure on $V$.  Using the topological identification 
of $V$ and $\widehat{V}$, we can view the Fourier transform as a function on 
$V$ by the formula

\begin{equation}\label{onv}
\widehat{f}(w) = \int_{V} \zeta_{-w}(v) f(v)\, dm(v) = \int_{V} [-w,v] f(v)\, dm(v)
\end{equation}
in the bracket notation of the pairing.

We will use repeatedly the following theorems(Plancherel, inversion and 
duality):
\begin{planch}
There is a normalization of the dual Haar measure $dm(\chi)$ on $\widehat{V}$ 
so that:
\begin{enumerate}
\item The Fourier transform extends to an isometry $L^2(V)\arw L^2(\widehat{V})$.
\item If both $f$ and $\widehat{f}$ are integrable, then for almost every $v\in V$
\begin{equation}\label{planc}
f(v) = \int_{\widehat{V}} \chi(v)\widehat{f}(\chi)\, dm(\chi).
\end{equation}
\item Every $v\in V$ defintes a unitary character of $\widehat{V}$ through 
the pairing $(v,\chi)\arw \chi(v)$ which furnishes a canonical topological 
isomorphism between $V$ and $\widehat{\widehat{V}}$.
\end{enumerate}
\end{planch}

The Schwartz-Bruhat space $\fs(V)$ is just the usual Schwartz space when $\ik$ is Archimedean; in 
the non-Archimedean case, it consists of compactly supported, locally 
constant functions on $V$.  The main properties of $\fs(V)$ are that its 
functions are dense in $L^2(V)$ and the Fourier transform furnishes a 
topological isomorphism $\fs(V) \simeq \fs(\widehat{V})$.  For more details 
about the Fourier analysis facts we will use \cite{Tai}.

Given a Schwartz function $\phi$ on $\widehat{V}$ and $f\in \alg_K$ define 

\begin{equation}\label{prodef11}
P_\phi (f) := \int_{V} \widehat{\phi}(x)\, (x\cdot f) \,dm(x)
\end{equation}
Here we use the formulation of \cite{Lang}, Chapter 11 for Banach-space valued integrals.  Because of the rapid decay of 
the Fourier transform $\widehat{\phi}$, $P_\phi(f)$ retains differentiability 
properties of $f$ and the inequality
\begin{equation}\label{lp0}
||P_\phi (f)||_p \leq ||\widehat{\phi}||_{L^1(V)}||f||_p\,,\quad 1\leq p \leq \infty
\end{equation}
shows that it is bounded on all the spaces we will consider.  Some structural 
properties of this operator (the case $\ik=\IR$ is worked out 
in \cite{HT}, chapter I; the general case has no new features regarding these 
properties) include

\begin{itemize}
\item $P$ is self-adjoint (with respect to the inner product of $\Hil$) for 
real $\phi$; more generally, $P^*_\phi = P_{\overline{\phi}}$.
\item Operator multiplication transforms to pointwise multiplication of 
functions: \begin{equation} \label{xpro0} \D{P_{\phi\psi}=P_\phi \circ P_\psi}
\end{equation}  This property plus linearity in the subscript shows that $P$ 
is a homomorphism from the pointwise algebra of Schwartz functions to self 
adjoint operators on $\Hil$.
\item Given the context of Section 2, for $g\in G$, \begin{equation}\label{conj} \sigma(g)P_\phi 
\sigma(g^{-1}) = P_{\phi(\rho(g^{-1})\cdot)}.\end{equation}  When $\phi$ is $K$
-invariant, this equation implies that $P_\phi$ commutes with the $K$-action 
and thus $K$-finite vectors are $L^2$-dense in the range of $P_\phi$.
\end{itemize}

Note that we will usually identify $\phi\in \fs(\widehat{V})$ with 
$\phi(\grz_\cdot)\in \fs(V)$.  With that identification, the action of $G$ in 
\eqref{conj} corresponds to the representation $\rho^*$ on $\fs(V)$, i.e. 
when we think of $\phi$ as afunction on $V$, we have 
\begin{equation}\label{conj2} \sigma(g)P_\phi \sigma(g^{-1}) = P_{
\phi(\rho^*(g^{-1})\cdot)}.\end{equation}

Convergence and limits involving $P$ are obtained using positivity: for 
$\phi\geq 0$, $P_\phi$ is a positive semidefinite operator.  To see this, 
simply use \eqref{xpro0} and self-adjointness:
\begin{equation*}
\langle P_\phi(f),f\rangle =\langle P_{\sqrt{\phi}}(P_{\sqrt{\phi}}(f)),f\rangle = \langle P_{\sqrt{\phi}}(f),P_{\sqrt{\phi}}(f)\rangle \geq 0
\end{equation*}
This way we see that $P_\phi \geq P_\psi$ and thus $||P_\phi||_2 \geq 
||P_\psi||_2$ when $\phi \geq \psi$.  Thus, if $\phi_j$ increase or decrease 
monotonically to a bounded function on $V$, the $P_{\phi_j}$ converge 
strongly to a bounded, self-adjoint operator on $\Hil$.  Although we will 
mostly deal directly with the $P_\phi$, since we cannot guarantee control on 
the $L^\infty$ norm for the limits in general, we will use them as a tool to 
simplify calculations.  Of course, under additional assumptions about the 
smoothness of the vectors, integration by parts in \eqref{prodef11} 
transforms the sequence into one which is $L^p$-convergent for any $p\geq 1$, 
but since we want to treat $K$-finite vectors that will not be necessarily 
smooth, we avoid the use of the limit operators as such.

Now let $S$ be a subset of $V$ with the property that its characteristic 
function $\chi_S$ can be pointwise approximated by a sequence of decreasing 
compactly supported Schwartz functions; we call such sets admissible and all 
sets we will deal with will be admissible.   In particular, we will be 
interested in the annuli $$\textrm{Ann}(s) := \{x\in V | s^{-1}< ||x||<s\}.$$
  Recall that the norm on $V$ is assumed $K$-invariant.  The characteristic function of each annulus 
$\chi_{\textrm{Ann}(s)}$ can be approximated by a sequence of smooth 
functions with the properties 

\begin{eqnarray*}
\phi_s^k &\equiv & 1\textrm{ on }\textrm{Ann}(s)\\
\textrm{supp}(\phi_s^k) &\subset &\textrm{Ann}\left(s+\frac{1}{k}\right)\\
\phi_s^k &\leq & \phi_s^l\textrm{  for }l\leq k
\end{eqnarray*}
From this definition, the sequence $P_{s,k} := P_{\phi_s^k}$ consists of 
positive, decreasing, self-adjoint (see \cite{HT} for the easy computation) 
bounded operators on $L^2(X)$ and thus has a strong limit for fixed $s$ as $k$
 tends to infinity which by \eqref{xpro0} is idempotent, since $\phi_s^k \arw 
\chi_{\textrm{Ann}(s)}$.  Note that the image under $P_s = \lim P_{s,k}$ of 
$L_0^\infty(X)$ is $L_0^2$-dense in $L_0^\infty$ since the $P_s$ form a 
system of projections that tends to the identity operator in $L_0^2(X)$ as $s$
 goes to infinity.

The properties of $P_\phi$ listed above imply trivially some important facts:
\begin{itemize}
\item If $\textrm{supp}(\phi) \subset S$, then 
\begin{equation}\label{stab} P_S(P_\phi)=P_\phi \end{equation}
\item If $S$ is invariant under rotations, then for any $g\in K$,
\begin{equation}\label{rotinv}  \sigma(g)P_S \sigma(g^{-1}) = P_S 
\end{equation}
\item By the previous property, when $S$ or $\phi$ are $K$-invariant, $P_S$ 
or $P_\phi$ commutes with the action of $K$ and thus $K$-finite vectors are 
dense in the range of $P_S$ or $P_\phi$.
\end{itemize}
The following ad hoc notation will be convenient: if $P_S(f)=f$ for some set 
$S$, we say that the 'spectral support' of $f$ lies in $S$; when $S$ is 
replaced in the subscript by a Schwartz function, the notion will refer to 
its support.  Intuitively, $P_S$ restricts the 'spectrum' of $f$ to lie in $S$
, so a function which is unaffected by this application is justified in being 
called 'spetrally supported in $S$'.  Note that $P_S(L^2)$ is a closed vector 
subspace of $L^2$ since the $P_S$ are norm bounded (for fixed $S$).

With this notation in hand, we can define explicitly the dense subspace of 
$\alg_K$ where we will bound the coefficients effectively.
\begin{dens}
We define $\mathcal{D}$ to be the union
$$ \mathcal{D} = \bigcup_{s>0,k>s^2} P_{\phi_s^k}(\alg_K)$$ and call it the 
space of spectrally bounded functions in $\alg_K$.
\end{dens}
It is easy to see that this space is $L^2$-dense in $\alg_K$.  Of course the 
specific choice $k>s^2$ is not important, we just need some leeway for 
approximations and we do not want $k$ to be too small as to cause problems 
with stretching annuli.

\section{Main Results}
\subsection{Behavior and bounds on projection operators}\label{lemmata}
First, we need a lemma on how pointwise multiplication of functions behaves 
with respect to the operators $P_\phi$.  Below we identify $\widehat{V}$ with 
$V$ and the two dual Haar measures by the isomorphism in 2.3 (compatibility 
in the computations below is guaranteed by \eqref{planc}).  Recall the 
notation $[u,z] = \zeta_u(z)$ for $u,z\in V$ (keep in mind the usual case 
$[u,z] = e^{i\langle u,z\rangle}$).  The result is

\begin{step1}\label{st1}
Let $\phi,\psi \in \fs(V)$ and $f,g\in L^2(X)$ be such that the pointwise 
(a.e.) product $P_\phi(f)P_\psi(g)$ is in $L^2(X)$.  Suppose $\omega \in 
\fs(V)$ is identically equal to one on 
$\textrm{supp}(\phi)+\textrm{supp}(\psi)$; then 
$P_\omega(P_\phi(f)P_\psi(g))=P_\phi(f)P_\psi(g)$.
\end{step1}
\begin{proof}
Compute:
\begin{eqnarray*}
& & P_{\omega} ( P_{\phi}(f)P_{\psi}(g))\\  &=& \int \widehat{\omega}(z)\int 
\widehat{\phi}(x)\rho(z+x)f \,dm(x)\int\widehat{ \psi}(y)\rho(z+y)g \mbox{ }\,
dm(y)\mbox{ }\,dm(z)\\
&=& \int \widehat{\omega}(z)\int \widehat{\phi}(x-z)\rho(x)f \,dm(x)
\int\widehat{ \psi}(y-z)\rho(y)g \mbox{ }\,dm(y)\mbox{ }\,dm(z)\\
&=& \iint \rho(x)f \mbox{ }\rho(y)g \int \widehat{\omega}(z) \widehat{\phi
}(x-z)\widehat{ \psi}(y-z) \,dm(z)\mbox{ }\,dm(x)\mbox{ }\,dm(y)
\end{eqnarray*}
Now expand the inner integral using the definition of the Fourier transform, 
valid for $L^1$ functions:
\begin{eqnarray*}
& & \int \widehat{\omega}(z) \widehat{\phi}(x-z)\widehat{\psi}(y-z) \,dm(z)\\
&=& \iiiint \omega(u_3)[-z,u_3]\phi(u_1)[-(x-z),u_1] \\ & &\quad\cdot\psi(u_2)
[-(y-z),u_2]\,dm(u_2) \,dm(u_1) \,dm(u_3) \,dm(z)\\
&=& \iint \phi(u_1)[-x,u_1]\psi(u_2)[-y,u_2] \\ & &\quad \cdot\left(\iint 
\omega(u_3)[-z,u_3]\,[z,u_1]\,[z,u_2]\,dm(u_3) \,dm(z)\right) \,dm(u_1) \,
dm(u_2)\\
&=& \iint \phi(u_1)[-x,u_1]\psi(u_2)[-y,u_2] \\ & &\quad \cdot\left(\int 
[z,u_1+u_2]\int \omega(u_3)[-z,u_3]\,dm(u_3) \,dm(z)\right) \,dm(u_1) \,
dm(u_2)\\ 
\end{eqnarray*}
The integral in the parentheses is simply
\begin{equation*}
\int [z,u_1+u_2]\widehat{\omega}(z)\,dm(z) = \omega(u_1+u_2) = 1
\end{equation*}
by Fourier inversion and the fact that $u_1\in \textrm{supp}(\phi),\,u_2\in 
\textrm{supp}(\psi)$.  Untangling the remaining integrals we get the required 
result.
\end{proof}
\begin{corr}
Let $\textrm{supp}(\phi)\subset S$ and $\textrm{supp}(\psi)\subset T$ for 
admissible sets $S$ and $T$.  Then 
\begin{equation}\label{mult} P_{S+T}(P_\phi P_\psi) = P_\phi P_\psi \end{equation}
\end{corr}
The relations \eqref{mult} and \eqref{stab} form the core of the main 
computation.

In the sequel, we will examine how restricting a unit (in the $L^2$-norm) $K$
-finite vector $f$ to the image of an approximate projection $P_\phi$ for 
suitable $\phi$ affects its norm.  This was accomplished in greater 
generality in \cite{Wang} from which we will draw notation and results, 
noting the places in that paper where they are treated.  The idea of 
estimating matrix coefficients (non-uniformly) by looking at the effect the 
representation has on their 'spectral support' (i.e. the smallest set with an 
approximate characteristic function $\phi$ such that $P_\phi(f)=f$) and then 
estimating norms of functions with restricted spectral support is a major 
theme in chapter 5 of \cite{HT}; the non-commutativity of $K$ in our setting 
increases the complexity of this method considerably.  However, the detailed 
analysis in \cite{Wang} allows one to carry it out effectively.

In order to state the second main lemma and principal ingredient for bounding 
norms of projected vectors, we need some additional concepts from \cite
{Wang}.  Recall the list of notations from Section 2 and assume that $\rho = \rho_1$ 
is irreducible with highest / lowest weights $\lambda, \varrho$.  For 
$\psi\in\Phi_1$, Let $\pi_\psi(v)$ be the projection of $v$ on the weight 
space $V_\psi$. 

Define the 'cones' 
\begin{eqnarray*}
\textrm{Cone}_1(c,s) = \{v\in V: ||\pi_\lambda(v)||\leq c\textrm{ and }||v||\geq s\},\\
\textrm{Cone}_2(c,s) = \{v\in V: ||\pi_\varrho(v)||\leq c\textrm{ and }||v||\geq s\}.\\
\end{eqnarray*}
See \cite[Proposition 6.1]{Wang} for the fundamental properties of these 
sets.  We will not use Proposition 6.1 itself here, but we will follow 
verbatim the computations in Proposition 7.1 which uses Proposition 6.1 in 
a crucial way.

Observe that the norm $\|\cdot\|_\infty$ on $V$ defined by 
$$\|v\|_\infty=\max_{\phi\in \Phi_1}\|\pi_\phi(v)\|$$ is equivalent to 
the given norm since $\textrm{dim}(V)<\infty$, so in particular $\|v\|_\infty 
\leq C\|v\|$; we will use this observation below.

\begin{compact}\label{compac}
Let $f$ be $K$-finite with $||f||_2=1$, $\textrm{dim}\langle K \cdot f\rangle 
= d_f$, $\mb{a}^i\in D^+$ for $i=1,\cdots,k$ ordered in \emph{increasing} 
$\pi_\lambda{\mb{a}^i}$ with sufficiently high norms\footnote{In the sense of 
\eqref{norms}.} depending \emph{only on the action}, $\textrm{Ann}(s)$ the 
annulus defined in Section 2.4 and 
$F_s\in \fs(V)$ with compact support inside the set $$X_1(\mb{a},s) = 
\textrm{Ann}(s)\cap 
\left(\sum_i\rho^*(\mb{a}^i)\left(\textrm{Ann}(s^{-1},s)\right)
\right).$$  Then for some positive 
$C$ independent of $\mb{a}$, $s$ and $f$ we have the bounds
\begin{equation}
||P_{F_s} (f) ||_2 \leq Cs^{\mfq} d^{\frac{1}{2}}_f |\sum_i \varrho ( \mb{a}^i )^{-1} |^{-\frac{\mfq}{2}}.
\end{equation}
Similarly, if the support of $F_s$ is in the set $$X_2(\mb{a},s) = 
\textrm{Ann}(s)\cap 
\left(\sum_i\rho^*(\frac{\mb{a}^i}{\mb{a}^k})\left(\textrm{Ann}(s^{-1},s)\right
) \right),$$ then as above
\begin{equation}
||P_{F_s} (f) ||_2 \leq Cs^{\mfq} d^{\frac{1}{2}}_f |\sum_i \lambda (\frac{
\mb{a}^k}{\mb{a}^i})|^{-\frac{\mfq}{2}}.
\end{equation}

\end{compact}
\begin{proof}

The proof is essentially contained in the proof of Proposition 7.1 of \cite{Wang}.  We 
indicate how to extract the relevant parts for our lemma and explain the 
correspondences.  All references to numbered sections will belong to \cite
{Wang}.  We will only examine the fist situation, since the second one is identical.

In the course of proving that proposition, the author examines (pages 31-38) 
a Schwartz function\footnote{The reason we use $F_s$ for the function rather 
than the usual Greek letters is to facilitate the comparison with the computation 
in \cite{Wang}.} $F_s$ and the 
projection of a unit vector with respect to that function, bounding
the Hilbert space norm of $$\widehat{\Pi}(F_s^\frac{1}{2})\eta $$ in the notation of that 
paper, where $$F_s = (a\omega)^{-1}\* (h_s\cdot g_s)^\frac{1}{2};$$ in our 
notation $\eta$ is $f$ and $\widehat{\Pi}(F_s^\frac{1}{2})$ is 
$P_{F_s^\frac{1}{2}}$, $\omega=1$ because we are only considering the 
positive Weyl chamber and the precise definition of $F_s$ in \cite{Wang} is 
irrelevant; in order to carry out the computations, we only need $F_s$ to be 
Schwartz and its support contained in one of the two cones defined above, for 
specific $0<c\ll 1$ and $s>0$.  There, it is claimed that how small $c$ needs 
to be depends on $s$; however, the only dependence of $c$ on $s$ that is 
necessary there is that $cs^{-1}<C$ with $C$ depending only on the action.  
Since our $c$ here is going to be of the form $c=sA$ where $A$ does not 
involve $s$, we see that in order to ensure $cs^{-1}<C$ all we need is to 
bound $A$ by a constant depending only on the action; this translates to the 
norm bounds on the $\mb{a}^i$ in the statement.  See p.27 of \cite
{Wang} for the specific requirements on $c$.

By the discussion above, the reductions from page 31 to 
page 34 carry over to our $F_s$, resulting in the situation where we want to 
bound $$\|\tilde{P}_{F_s}\tilde{f}\|_2$$ where $\tilde{P}$ is the approximate 
projection operator for the regular representation of $K\ltimes V$ on $\Hil$ 
and $\tilde{f}$ is a $K$-invariant unit norm vector.  At that point we use 
the containment of the supports.  In our case, observe that our 
$X_1(\mb{a},s)$ is contained in the set $$E_1 = \{v\in 
\textrm{Ann}(s):s^{-1}\leq \|\sum_{i=1}^k \rho^*(\mb{a}^i)v \|\leq s \} $$
which becomes, after writing $v$ in terms of the weights, applying $\rho^*$ 
to each coordinate in $V_\psi$ and switching summations
$$E_1 = \{v\in \textrm{Ann}(s):s^{-1}\leq \|\sum_{\psi\in 
\Phi_1}(\sum_{i=1}^k \psi(\mb{a}^i)^{-1})\pi_\psi(v) \|\leq s \}. $$ Note the 
inverses because we are decomposing with respect to the $\Phi_1$ of $\rho$.  
Using the equivalence of norms $\|\cdot\|$ and $\|\cdot\|_\infty$ we see that this 
set is contained in $$S_1 =  \{v\in \textrm{Ann}(s): \|\pi_\varrho(v)\|\leq Cs |\sum_{i=1}^k 
\varrho(\mb{a}^i)^{-1}|^{-1} \}.$$

Since $\mb{a}^i\in D^+$, for large $\min_{1\leq i\leq k}|\mb{a}^i|$ (in any 
norm on $G$) the coefficient on the right hand side of the definition on $S_1$
is going to be small.  In particular, $S_1$ will be contained in the cone
$$\textrm{Cone}_1( Cs |\sum_{i=1}^k \varrho(\mb{a}^i)^{-1}|^{-1},s^{-1}).$$  
Having this containment, the argument from pages 34-37 goes through without 
change leading to the desired conclusion analogous to (7.22), (7.23) there.

\end{proof}
Since admissible sets can be approximated from above by Schwartz functions 
and the operators $P_{\cdot}$ are monotone, we get the corollary 
\begin{corrr}\label{corl}
With notation as in Lemma \ref{compac} and $S$ an admissible set contained in 
one of the $X_i(\mb{a},s)$, we have the corresponding bound from that lemma 
for $\|P_S(f)\|_2$.
\end{corrr}

\subsection{Main theorem}\label{theorems}
Now consider $k\geq 2$ distinct elements $\mb{a}^i \in D^+$, $i=1\cdots k$ as above and $k+1$ functions $f_i\in 
\mathcal{D}, i=0\cdots k$ of $L^2$ norm $1$.  Order the $\mb{a}^i$ in increasing highest weight 
valuations and define $\mb{a}^0=\textrm{I}$.  We want to bound the correlation integral
\begin{equation}\label{correl}
 \int_X f_0(x)\,\mb{a}^1\cdot f_1( x)\,\cdots \,\mb{a}^k \cdot f_{k}(x)\,d\mu(x).
\end{equation}
Since $P_{s,k}(f)\arw f$ for any $f\in \alg$ and we only have finitely many $f_i$, 
we can assume that all $f_i$ are in the image of $P_{s,l}$ for some $s,l$ 
(and thus certainly in the image of $P_{s'}$ where $s' = s + \frac{2}{l}$).  

For notational convenience, abbreviate $\textrm{Ann}(s)$ by $(s)$ and denote 
its image under $\rho^*(\mb{a}^i)$ simply by 
$\mb{a}^i(s)$.  We will also denote the action of $\rho^*(\mb{a}^i)$ on the 
$\phi_s^k$ defined above simply by $\mb{a}^i(s,k)$.

\begin{central}\label{centralt}
Let $\mb{a}^i, f_i, s$ be as above.  Let 
$$d_i = \dim\langle K\cdot f_i\rangle.$$  There exists a positive constant 
$C'$ independent of the $f_i$ such that if 
\begin{equation}\label{a-ass}\min(\min_{i=0,\cdots,k-1}|\lambda(\frac{\mb{a}^k}{\mb{a}^i})|,\min_{i=1,\cdots,k}|\varrho(\mb{a}^i)^{-1}|)>C',
\end{equation} we have the bound
\begin{eqnarray}
& & \int_X f_0(x)\,\mb{a}^1\cdot f_1( x)\, \cdots \,\mb{a}^k \cdot f_{k}(x)\,d
\mu(x)\nonumber \\
& & \qquad \leq  C s^{2\mfq} d^{\frac{1}{2}}_0 d^{\frac{1}{2}}_k\left|\sum_
{i=1}^{k-1} \lambda(\frac{\mb{a}^k}{\mb{a}^i}) \right|^{-\frac{\mfq}{2}}
\left|\sum_{i=1}^k \varrho(\mb{a}^i)^{-1} \right|^{-\frac{\mfq}{2}}\label
{estimate}
\end{eqnarray}
where $C$ only depends on the $L^\infty$ norms of the $f_i$.
\end{central}
\begin{proof}
The correlation can be written as 
\begin{equation}
\int_X P_{s,l}(f_0)\,\mb{a}^1 \cdot P_{s,l}(f_1)\,\cdots \mb{a}^k\cdot P_{s,l}(f_k)
\end{equation}
which by \eqref{conj} becomes
\begin{equation}
\int_X P_{s,l}(f_0)\,P_{\mb{a}^1 (s,l)}(\mb{a}^1\cdot f_1)\,\cdots \,P_{\mb{a}^k (s,l)}(\mb{a}^k\cdot f_k).
\end{equation}
We now use Lemma \ref{st1} repeatedly to conclude that $$P_{\mb{a}^1 
(s,l)}(\mb{a}^1\cdot f_1)\cdots P_{\mb{a}^k (s,l)}(\mb{a}^k\cdot f_k) \in 
P_\Sigma(\alg)$$ where $\Sigma$ is the iterated sum set $\sum_{i=1}^k 
\mb{a}^i(s')$; thus in particular if $$z := P_{\mb{a}^1 (s,l)}(\mb{a}^1\cdot 
f_1)\cdots P_{\mb{a}^k (s,l)}(\mb{a}^k \cdot f_k)$$ then $P_\Sigma(z)=z$.  
Thus the integral above becomes
\begin{equation}
\int_X P_{s,l}(f_0)P_\Sigma(z)
\end{equation}
Now $P_\Sigma$ is an orthogonal projection so we can transfer $P_\Sigma$ from 
$z$ to $P_s(f_0)$, getting $$P_\Sigma(P_{s,l}(f_0)) = P_{\chi_\Sigma 
\phi_s^l}(f_0).$$  Here we are abusing notation a little bit, since the last 
expression need not be a bounded function; we will take this shortcut to mean 
that we have an arbitrary Schwartz function $\phi$ dominating the function 
$\chi_\Sigma$ and we are applying $P_\phi$ to both terms of the 'inner 
product'; the rightmost term is unaffected, while the leftmost has spectral 
support approximately equal to that of $\chi_\Sigma \phi_s^l$ since $\phi$ is 
arbitrary and the support of $\chi_\Sigma \phi_s^l$ is easily seen to be an 
admissible set (also see Corollary \ref{corl}).
  Thus the integral becomes 
\begin{eqnarray*}
\int_X P_{{\chi_\Sigma \phi_s^l}}(f_0)z &=& \int_X P_{{\chi_\Sigma \phi_s^l
}}(f_0)P_{\mb{a}^1 (s,l)}(\mb{a}^1\cdot f_1)\cdots P_{\mb{a}^k (s,l)}(\mb{a}^k
\cdot f_k)\\
&=& \int_X P_{{\chi_\Sigma \phi_s^l}}(f_0)\mb{a}^1 \cdot P_{s,l}(f_1)\cdots 
\mb{a}^k \cdot P_{s,l}(f_k)
\end{eqnarray*}
Write $U_0 := \chi_\Sigma \phi_s^l$ and apply $(\mb{a}^k)^{-1}$ to all terms 
of the integral, giving
\begin{equation}
\int_X (\mb{a}^k)^{-1}\cdot P_{U_0}(f_0)P_{s,l}(f_k)\prod_{i=1}^{k-1}(\mb
{a}^k)^{-1}\cdot \mb{a}^i\cdot P_{s,l}(f_i)
\end{equation}
By unitarity, the value of the integral is not affected.  So, now we can 
repeat the reasoning above, summing the indices for all factors except for 
$P_{s,l}(f_k)$, and conclude that this integral is equal to
\begin{eqnarray}
& & \int_X (\mb{a}^k)^{-1}\cdot P_{U_0}(f_0)\, P_{s,l}(f_k)\,\prod_{i= 1}^{k-1} (\mb{a}^k)^{-1}\cdot\mb{a}^i\cdot P_{s,l}(f_i)\nonumber \\
&=& \int_X P_{s,l}(f_k) P_{\Sigma_k}(z_k)\\ & = &\int_X P_{\chi_{\Sigma_k}\phi_s^l}(f_k)z_k\nonumber \\
 &=& \int_X P_{U_0}(f_0) \mb{a}^k \cdot P_{U_k}(f_k)\prod_{i=1}^{k-1} \mb{a}^i\cdot P_{s,l}(f_i)\label{fini}
\end{eqnarray}
where $$\Sigma_k= \sum_{j=0}^{k-1} (\mb{a}^k)^{-1}\cdot \mb{a}^j(s'),$$ $$z_k 
= \prod_{i= 0}^{k-1} (\mb{a}^k)^{-1}\cdot\mb{a}^i\cdot P_{s,l}(f_i)$$  and 
$U_k = \chi_{\Sigma_k} \phi_s^l$.  

Denote by $U_0$ and $U_k$ respectively also the supports of the corresponding 
functions (which, note, are bounded above by $1$ and thus by the 
characteristic functions of the supports).  We can now immediately apply Lemma \ref
{compac} with $U_0$ and $U_k$ in the place of the two situations for $F_s$ 
considered there, bounding

$$\|P_{U_0}(f_0)\|_2\textrm{ and }\|\mb{a}^k \cdot P_{U_k}(f_k)\|_2=\|P_{U_k}(f_k)\|_2.$$

In order to finish the proof, we simply bound \eqref{fini} by the $L^\infty$ 
norms of the functions $f_i$ for $i\neq 0,k$ and then use Cauchy's inequality 
on the two remaining terms to finish the proof.
\end{proof}
\subsection{Examples}  In this section we see what the bounds obtained above 
mean for two particular cases of semidirect products.  The choice of these 
examples is not arbitrary: these groups will occur in Section \ref
{simple_groups} as subgroups (locally) of split simple groups of higher rank.

First consider the case of $\mathfrak{G}=\textrm{SL}(2,\ik)\ltimes \ik^2$ where the action 
is the standard matrix action on $2$-vectors.  The action is irreducible and 
there are only two weights; the roots of $\textrm{SL}(2,\ik)$ are $$\textrm{diag}(a,a^{-1}) 
\arw a^{\pm 2}$$ and the weights of the standard representation on $\ik^2$ are 
$$\textrm{diag}(a,a^{-1}) \arw a^{\pm 1}$$ with the obvious weight spaces 
$V_1 = \{(v_1,0)\in \ik^2\}$ and $V_2 = \{(0,v_2)\in \ik^2\}$.  Take the weight $$\textrm{diag}(a,a^{-1}) 
\arw a$$ to be positive, so this is the highest weight.  The highest 
weight space is one dimensional, so the 
exponent $\mathfrak{q}$ defined in Section \ref{defs} is in our case $1$.  Write 
\begin{displaymath}
 \mb{a}^i = \left( 
 \begin{array}{cc} 
 a_i & 0\\
 0 & a_i^{-1}
 \end{array}\right)
 \end{displaymath}
 for $k+1$ elements of the positive Weyl chamber with $1=a_0<a_1< \cdots < 
a_k$ and if $i>j$, $\frac{a_i}{a_j}>C_0$ for some $C_0$ depending on the 
action of $\mathfrak{G}$ on $X$.  Applying the preceding discussion to the 
hypotheses of Theorem \ref{centralt}, we get

\begin{sl_cor}\label{slc}
In the setting of Theorem \ref{centralt} and 
$\mathfrak{G}=\textrm{SL}(2,\ik)\ltimes \ik^2$ with the standard action, we 
get the bound 
\begin{eqnarray*}
& & \int_X f_0(x)\,\mb{a}^1\cdot f_1( x)\, \cdots \,\mb{a}^k \cdot f_{k}(x)\,d
\mu(x) \\
& & \qquad \leq  C s^{2} d^{\frac{1}{2}}_0 d^{\frac{1}{2}}_k\left|\sum_
{i=0}^{k-1} \frac{a_k}{a_i} \right|^{-\frac{1}{2}}\left|\sum_{i=1}^k {a}_i 
\right|^{-\frac{1}{2}}.
\end{eqnarray*}
\end{sl_cor}
Note how in the case $k=1$ we recover the bound from Chapter 5 of \cite{HT}.

For the second example, consider the action of $\textrm{SL}(2,\ik)$ on its 
Lie algebra over $\ik$, denoted simply by $\mathfrak{g}$ and being equivalent 
to $S^2(\ik^2)$, the second symmetric power of $\ik^2$ (in the case 
$\textrm{char}(\ik)=0$ that we are considering).  The weights and weight 
spaces in this case coincide with the roots and the highest weight space (pick 
$\textrm{diag}(a,a^{-1})\arw a^2$ as positive) is again one dimensional.  
Therefore, by the same procedure as above, we have:

\begin{ad_cor}\label{adc}
In the setting of Theorem \ref{centralt} and 
$\mathfrak{G}=\textrm{SL}(2,\ik)\ltimes \mathfrak{g}$ with the adjoint action on the Lie algebra, we 
get the bound 
\begin{eqnarray*}
& & \int_X f_0(x)\,\mb{a}^1\cdot f_1( x)\, \cdots \,\mb{a}^k \cdot f_{k}(x)\,d
\mu(x) \\
& & \qquad \leq  C s^{2} d^{\frac{1}{2}}_0 d^{\frac{1}{2}}_k\left|\sum_
{i=0}^{k-1} \left(\frac{a_k}{a_i}\right)^2 \right|^{-\frac{1}{2}}\left|\sum_{i=1}^k (a_i)^2 
\right|^{-\frac{1}{2}}.
\end{eqnarray*}
\end{ad_cor}

 \section{Extending the bound}\label{uniformity}
 
 In this section we examine the limits of our method and describe the 
broadest class of subspaces of $L^2_0(X)$ where we can get an effective 
estimate depending only on the action and various norms of the functions in 
the correlation integral.
\subsection{Beyond spectral restriction}
 Observe that whenever we have an estimate of the form \begin{equation}\label{hope}\|P_{(s)}(f)-f\|_2\leq C 
 \|f\|'s^{-A}\end{equation} for all $s>0$, $C$ and $A$ independent of $f$ and $\|\cdot\|'$ an 
 appropriate norm, we can use a $2\gre$ argument plus the uniform H\"older 
 inequality to eliminate $s$:

 \begin{eqnarray*}
 |\int_X f_0\cdots \mb{a}^k\cdot f_k\,dx| &\leq & \sum_{i=0}^k \prod_{j\neq i}\|f_j\|_\infty \|P_{(s)}(f_i)-f_i\|_2
 \\ &+& |\int_X P_{(s)}(f_0)\cdots \mb{a}^k\cdot \mb{a}^i \cdot P_{(s)}(f_k)\,dx|\\
 & \leq & C\left(\sum_{i=0}^k \|f_i\|'\prod_{j\neq i}\|f_j\|_\infty\right)s^
 {-A} + C's^{2\mfq}d_0^{\frac{1}{2}} d_k^{\frac{1}{2}}\mathfrak{R}(\mb{a})
 \end{eqnarray*}
 where $\mathfrak{R}(\mb{a})$ is the factor in \eqref{estimate} depending on 
 $\mb{a}$.  Choosing $s = \mathfrak{R}(\mb{a})^\gre$ and optimizing for $\gre$ 
 to get the best overall exponent, we can get a uniform bound for all $K$-finite 
 vectors with finite $\|\cdot\|'$-norm.
 
 In order to axiomatize this estimate, we introduce, for each $A>0$, the norms
\begin{eqnarray}
\|f\|_{-,A} &=& \sup_{0<s<\infty}s^{-A}\|P_{B(0,s)}(f)\|_2,\\
\|f\|_{+,A} &=& \sup_{0<s<\infty}s^A\|P_{(B(0,s))^\complement}(f)\|_2,\\
\|f\|_{\pm,A} &=& \|f\|_{-,A}+\|f\|_{+,A},\label{pmnorms}
\end{eqnarray}
where $B(0,s)$ is the ball centered at the origin of radius $s$ with respect 
to the given norm on $V$.  Define $\alg^A_K$ to be the subspace of $\alg_K$ 
where $\|f\|_{\pm,A}$ is finite.  Note that for each $A>0$, $$P_{(s)}(\alg_K) \subset 
\alg^A_K$$ for all $s>0$ and thus $$\mathcal D\subset \alg^A_K.$$  
Furthermore, note that $\|\cdot\|_{\pm,A}$ is comparable to $\|\cdot\|_2$ on 
each $P_{(s)}(\alg_K)$, but not on $\mathcal D$ or $\alg^A_K$.  However, taking 
radial functions $\psi_{A,\gre}(r)$ that equal $r^{A}$ on $B(0,1-\gre)$, 
$r^{-A}$ on $B(0,1+\gre)^\complement$, are smooth and equal to $1$ on 
$\displaystyle{B(0,1+\gre/2)\setminus B(0,1-\gre/2)}$ we see that $$\bigcup_{A>0}\alg^A_K$$ is 
$L^2$-dense in $\alg_K$ (because $P_{\psi_{A,\gre}}(f)\arw f$ as $\gre, A\arw 0$).

The spaces $\alg^A_K$ form the broadest category of spaces where our method 
extends to give effective bounds.  This should be understood in the sense that if 
we use as inputs only Theorem \ref{centralt} and the basic structure of the 
projection operators $P$, we definitely need an estimate like \eqref{hope} to 
remove the dependence on $s$.  In this broadest class, our main result takes the form
\begin{broadest}\label{broad}
Let $f_i\in \alg^A_K$ of $L^2$-norm $1$, $\mb{a}^i$ and $d_i$ as in Theorem \ref
{centralt}.  Under the assumption \eqref{a-ass}, we have the bound
\begin{eqnarray}
& & \int_X f_0(x)\,\mb{a}^1\cdot f_1( x)\, \cdots \,\mb{a}^k \cdot f_{k}(x)\,d
\mu(x)\nonumber \\
& & \qquad \leq   d^{\frac{1}{2}}_0 d^{\frac{1}{2}}_k\left(\prod_{i=0}^k \|f_i\|_\infty\|f_i\|_{\pm,A}\right)\left(\left|\sum_
{0=1}^{k-1} \lambda(\frac{\mb{a}^k}{\mb{a}^i}) \right|
\left|\sum_{i=1}^k \varrho(\mb{a}^i)^{-1} \right|\right)^{-\frac{A\mathfrak{q}}{2(A+2\mathfrak{q})}}\label
 {estimate2}.
\end{eqnarray}
\end{broadest}
 \subsection{Beyond $K$-finiteness}\label{beyond}
 For smooth actions in the non-Archimedean case, $K$-finiteness is automatic 
 since the stabilizer of any smooth vector is open (and thus of finite index 
 in the maximal compact (open) subgroup of $G$).  In the Archimedean case, however, 
 smooth vectors can be far from $K$-finite.  In order to pass from $K$-finite 
 vectors to a larger class we can use the argument given in \cite[Theorem 
3.1]{KaSp}; we consider Sobolev vectors $f$, i.e. functions on which the $m$
 -fold action of the Casimir element $\Omega$ of the Lie algebra of $G$ is 
 defined and the following norm is finite(Theorem 4.4.2.1 in \cite{War} does 
not require any restrictions on the norm):  $$\|f\|_{\infty,A,m} := 
\max(\|\Omega^m(f)\|_\infty, 
\|\Omega^m(f)\|_{\pm,A})<\infty. $$  For such $f$ with $\|f\|_{\infty,A,m}< 
\infty$ for sufficiently large $m$, we mimic the computation in \cite{KaSp} line for 
line: 
 for each projection of each $f_i$ onto $K$-isotypic components we get the 
 estimate \eqref{estimate2} which only requires $\|\mb{a}^i\|$ to be 
 large enough depending only on the action\footnote{Actually, even if the 
magnitude of the acting elements depended on $s$, the argument still goes 
through for the following reason: the operators $P_{(s)}$ commute with the 
action of $K$ and $K$-isotypic components are given by a convolution over $K$; 
therefore, the $K$-isotypic components of $P_{(s)}f$ are the $P_{(s)}$-images 
of $K$-isotypic components of $f$.  Since $s$ does not change 
throughout the summation, any possible requirement on the magnitude of the 
acting elements in terms of $s$ is preserved.}.  Note that each $K$
 -type of a smooth function satisfies the conditions of Theorem \ref{broad} 
 because of the form of the projection (convolution with a character of a 
 compact Lie group).  In place of the Cauchy-Schwartz inequality in \cite{KaSp} we use H\"older's inequality 
 and the different powers of the $\textrm{dim}(\mu)$ that occur do not cause 
 the sum over $\widehat{K}$ to diverge because they are still beaten by a 
 polynomial of greater degree for large enough $m$ (\cite[Lemma 
 4.4.2.3]{War}).  

Write $\alg^A$ for the subspace of $\alg$ defined by the finiteness of \eqref{pmnorms}.  The result becomes
 \begin{arch1}\label{esi2}
 For $f_i\in \alg^A$ with $\|f\|_{\infty,A,m}<\infty$ for $m>m_0$ and $\|f_i\|_2=1$, 
 $\|\mb{a}^i\|>C'$ with $C'$ and $m_0$ depending only on the action, we have 
 the bound
 \begin{eqnarray}
 & & \int_X f_0(x)\,\mb{a}^1\cdot f_1( x)\, \cdots \,\mb{a}^k \cdot f_{k}(x)\,d
 \mu(x)\nonumber \\
 & & \qquad \leq \left(\prod_{i=0}^k \|f\|^2_{\infty,A,m}\right)\left(\left|\sum_
{0=1}^{k-1} \lambda(\frac{\mb{a}^k}{\mb{a}^i}) \right|
\left|\sum_{i=1}^k \varrho(\mb{a}^i)^{-1} \right|\right)^{-\frac{A\mathfrak{q}}{2(A+2\mathfrak{q})}}\label{estimate3}
 \end{eqnarray}
 \end{arch1}
 It is possible to replace the $L^\infty$ norm in 
the definition of $\|f\|_{\infty,A,m}$ with an $L^2$ norm by increasing $m$ 
(depending on the space $X$).  Furthermore, one 
 can replace the $L^\infty$ norms in all the arguments from Theorem \ref
{centralt} and beyond by various $L^p$ norms depending on the number of 
functions in the correlation, but we chose 
 the simpler and more uniform presentation using $L^\infty$.

The spaces $\alg^A$ and $\alg^A_K$ contain natural classes of 
functions.  Informally, if we think of $\mathcal{D}$ as the space of 
functions with finite Fourier series supported away from zero, we should think of $\alg^A$ as a space 
of functions with rapidly convergent Fourier series (or equivalently, 
in the Archimedean case, sufficiently smooth functions).  We require this rapid
decay at zero as well as infinity.  The finiteness 
of the norm $\|f\|_{-,A}$ is a quantitative version of the 
assertion that there are no $\sigma(V)$-invariant functions.  Note that this 
finiteness holds at least for functions in the 
range of $P_\phi$ for any $\phi \in C^1(V)$ for Archimedean $V$ and holds 
always for smooth vectors in the non-Archimedean case.  Furthermore, it is 
not hard to see that in 
the Archimedean case, if the matrix coefficient $\langle \sigma(v)\cdot 
f,f\rangle$ is smooth with sufficiently many derivatives bounded on $V$, then 
$\|f\|_{+,A}$ is finite; the permissible $A$ depend on the degree of smoothness.

\section{Simple groups of higher rank}\label{simple_groups}

\subsection{Semidirect products inside simple groups}\label{sdirect} In this 
section we describe how to use the results obtained so far to get 
effective multiple mixing in simple split groups of higher rank (and by 
extension for semisimple groups with finite center all of whose almost simple 
factors have higher rank).  We do this by locating sufficiently many 
semidirect products inside the simple group to which we can apply the main 
results.  We keep notation from previous sections when referring to the 
functions $f_i$ in the definition of the multiple correlation, the Cartan 
elements $\mb{a}^i$ etc.  A detailed description of the reduction process 
outlined below can be found in \cite{Oh} 
using a slightly different language.

In our setting, we are given a simple algebraic group split over $\ik$ of rank greater than or equal 
to $2$, a 
maximal $\ik$-split torus $D$, root system $\Phi = \Phi(G,D)$ and ordering 
$\Phi^+$; consider a 
mixing action $\sigma$ of $G$ on a standard probability space $(X,\mu)$.  We 
want to apply the results above to bound correlation coefficients 
for the Weyl chamber action $\sigma(D^+)$ on $X$.  In order to achieve this, following the 
proof of Proposition (1.6.2) in \cite{Mar} we do the following: 
given a positive simple root $\omega \in \Phi^+$, we choose another root 
$\omega'$ that is not orthogonal to $\omega$.  Then from the Dynkin diagram 
this pair of roots corresponds to either an $A_2$ system or a $C_2$ system, so we get a surjective 
morphism $\textrm{SL}(3)\arw \langle 
U_{\pm\omega},U_{\pm\omega'}\rangle=:G_\omega$ (respectively $\textrm{Sp}(4)\arw \langle 
U_{\pm\omega},U_{\pm\omega'}\rangle=:G_\omega$ ) with 
finite central kernel.  Furthermore, if  
$\omega$ corresponds to the root $\bar{\omega}$ (in either group) then its 
kernel in $D\cap G_\omega$ corresponds to the kernel 
in the diagonal $A\subset \textrm{SL}(3)$ (resp. $A\subset \textrm{Sp}(4)$ ) 
of $\bar{\omega}$.

So far, we have associated to each root of $G$ a corresponding torus in the 
uppermost copy of $SL(2)$ inside one of our rank $2$ groups, along with an 
isogeny to $G_\omega < G$ that carries that torus to the image of the one 
parameter group associated to $\omega$.  In both groups, this $SL(2)$ copy 
comes with an action on a vector space $V$ ($\ik^2$ in the $\textrm{SL}(3)$ 
case, $\ik^3$ in the $\textrm{Sp}(4)$ case; see \cite{BeHa}, Sections I.1.4 - 
I.1.5 for details), i.e. we have an isogeny from $\textrm{SL}(2,\ik)\ltimes V$ to its 
image in $G_\omega$ with the positive diagonal in $\textrm{SL}(2)$ mapping to 
a one parameter semigroup in the positive Weyl chamber of $G$.   Decompose 
$D = \textrm{ker}(\omega)D_\omega$ so that $D_\omega$ corresponds to the 
diagonal of $\textrm{SL}(2)$.  The image of the $\textrm{SL}(2)$ in $G_\omega$
 (i.e. the group generated by $U_{\pm\omega}$, see \cite[Chapter XI]{Hum}) 
commutes with $\textrm{ker}(\omega)$ by observing that their Lie 
algebras commute.  The greatest significance of this observation is that any maximal 
compact subgroup $K_\omega$ of that image commutes with $\textrm{ker}(\omega)$
. This fact plus the $K$-finiteness of the $f_i$ imply the $K_\omega$
-finiteness of the translates of the $f_i$ by elements in 
$\textrm{ker}(\omega)$; note that these translates are no longer necessarily 
$K$-finite.

The second consideration we need involves the divergence of the acting 
elements $\mb{A}\in D^+$ to infinity in the group $G$.  Our theorems, being 
about semidirect products, were formulated so that the divergence hypothesis 
takes the form given in inequality \eqref{norms}.  In the group $G$, the 
quickest way to link divergence to \eqref{norms} is form a norm $\|\cdot\|$ on $D$ by transporting the norm 
from the Lie algebra (normed as a $\ik$-vector space) via a 
complete collection of one parameter subgroups spanning $D$ defined over $\ik$
.  Then the correspondence between the norms in $V$ that appear in the bound 
\eqref{estimate} and norms in $G$ is given by a positive constant in the 
exponent, using the fact that all norms in a finite dimensional $\ik$-space 
are equivalent.

Now suppose $\mb{A}\in D^+$ acts on a 
function $f$ in the action defined above; we can write $\mb{A}=\mb{a}\mb{C}$ 
with $\mb{a}$ being (the image in $D^+$ of) an element of the diagonal group of 
$\textrm{SL}(2)$, $\mb{C}$ centralized by the maximal compact of that 
$\textrm{SL}(2)$, $f$ affords an action of $\textrm{SL}(2)\ltimes V$ with 
no invariant vectors for $V$ on $L^2(X)$ (mixing descends to subgroups 
and an isogeny has finite kernel, so we get no invariant vectors for the $V$ 
factor).  Doing this for all terms in a correlation $\mb{A}^i\cdot f_i$ we 
get $K$-finite vectors $f$ for an action of $\textrm{SL}(2,\ik)\ltimes V$ 
and we can thus apply the results obtained so far to that context.  In order 
to ensure that at least some tuple $(\mb{a})$ among the various choices goes 
to infinity (by the remarks of the previous paragraph, we can talk about 
distance either in terms of norm on $G$ or in terms of the weights of ratios 
of the $\mb{a}$), we use the 
assumption that the original tuple $\mb{A}$ goes to infinity and for each 
root there will be at least one other root (perhaps the same) so that the 
corresponding one parameter subgroups have a large ratio; if the original 
tuple goes to infinity, this process can be carried out for all $f_i$ and at 
this point we can apply our results for $\textrm{SL}(2,\ik)\ltimes V$ 
mentioned in Corollaries \ref{slc} and \ref{adc} to bound the correlation.

If $G$ is a semisimple split group without compact factors and all its simple 
factors satisfy the properties posited above, we can extend the discussion above to this 
setting.  Note that our methods cannot extend to groups of 
rank $1$ and their products.  To see this, observe that for the $2$-correlations 
our results essentially reduce to those of \cite{Wang} with some 
modifications or \cite{HT}, Chapter V, and \emph{prove} effective mixing 
for mixing actions of each group they apply to.  Since groups of rank $1$ by themselves do not 
need be effectively mixing (take for example the decay for the complementary series in \cite
[Chapter V, Proposition 3.1.5]{HT}), the method is inapplicable to that case.

 \subsection{From $\textrm{SL}(n,\ik)$ to $\textrm{SL}(2,\ik)\ltimes \ik^2$}\label{SLN}
 Let's illustrate the procedure above in the simple context of $\textrm{SL}(n)$.
 The first step in the derivation of our bound is to drop from 
 $\textrm{SL}(n,\ik)$ to the semi-direct product $\textrm{SL}(2,\ik)\ltimes 
 \ik^2$.  In this section, denote by $K(n)$ the maximal compact of 
 $\textrm{SL}(n)$.  We will not describe the reduction in detail here, since 
 it can be found in  detail in pages 227 - 228 of \cite{HT} in the archimedean 
 case, and the non archimedean case is similar.  We simply give a pictorial 
 sketch of the reduction:
 \begin{itemize}
 \item $\textrm{SL}(2,\ik)\ltimes \ik^2$ embeds in $\textrm{SL}(n,\ik)$ in the 
 following ways (all elements not depicted are zero, $1$ on the diagonal).

\footnotesize 
 \begin{displaymath}
 \textrm{SL}(2,\ik)\ltimes \ik^2 \ni \left(\begin{array}{ccc}
 a & b & x\\
 c & d & y\\
 0 & 0 & 1
 \end{array}\right)
 \end{displaymath}
 \begin{displaymath}\stackrel{j,\,l}{\longrightarrow}\qquad  \stackrel{\textrm{SL}_{jl}}{\left(\begin{array}{ccccccc}
  &  & 1j & & 1l & \\ 
  &\ddots & \vdots & & \vdots & & \\ 
 j1 & & a & \cdots & b & x & \\ 
 & & \vdots &\ddots & \vdots &\vdots & \\ 
 l1 &  & c & \cdots & d & y & \\ 
 & &  \vdots &  &\vdots & 1 & \\
 & &   &  & &  & \ddots \\
 \end{array}\right)}
 \end{displaymath}
\normalsize
 \item  In our bound, we have diagonal elements acting.  We can extract a 
 single element in the diagonal group of $\textrm{SL}(2)$ as follows:
\footnotesize 
 \begin{displaymath}
 \mb{a} = \left( 
 \begin{array}{ccccc} 
 \ddots &  &  &  & \\ 
  & a_j &  & \\ 
  &  &\ddots & & \\
  &  &  & a_l & \\
  &  & &  & \ddots
 \end{array}\right) \\ 
 \end{displaymath}
 \begin{displaymath}
 =\stackrel{\hat{\mb{a}}}{ \left( 
 \begin{array}{ccccc} 
 1 &  &  &  & \\ 
  & b &  & \\ 
  &  &\ddots & & \\
  &  &  & b^{-1} & \\
  &  & &  & 1
 \end{array}\right) }
 \stackrel{\mb{a}'}{\left( 
 \begin{array}{ccccc} 
 a_1 &  &  &  & \\ 
  & c &  & \\ 
  &  &\ddots & & \\
  &  &  & c & \\
  &  & &  & a_n
 \end{array}\right) }
 \end{displaymath}
\normalsize
 where  $b$ and $c$ are defined in the archimedean case by $\D{c = 
 \sqrt{a_{j}a_{l}}}$ and $\D{b = \frac{a_j}{c}}$; in the non-archimedean case, 
 we have the same definition if the difference of exponents of $q$ is even, 
 otherwise we will compensate by adding and subtracting a 
 $$\textrm{diag}(\cdots, q^{\frac{1}{2}},\cdots,q^{-\frac{1}{2}},\cdots)$$ to 
 bring the matrix entries back in $\ik$; obviously, the decay is not affected 
 by this tweak.  

 Note that $\mb{a}'$ commutes with the specific copy of $K(2)$ inside 
 $\textrm{SL}_{jl}$ so if $f$ is a $K(n)$-finite function, the new function 
 $\mb{a}'\cdot f$ is $K(2)$-finite for the action of that $\textrm{SL}_{jl}$ 
 copy.    
 \item Therefore, in order to derive estimates on the initial correlation 
 integral, we are led to consider $(X,\mu)$ with a mixing action of $G:= 
 \textrm{SL}(2,\ik)\ltimes \ik^2$ and bound integrals of the form
 \begin{equation}
  \int_X f_0(x)\,\mb{a}^1\cdot f_1( x)\,\cdots \,\mb{a}^k \cdot f_{k}(x)\,d\mu(x)
 \end{equation}
  where
 \begin{displaymath}
 \mb{a}^i = \left( 
 \begin{array}{cc} 
 a_i & 0\\
 0 & a_i^{-1}
 \end{array}\right)
 \end{displaymath}
 and the $f_i$ are bounded, zero mean functions satisfying $K$-finiteness 
 properties inherited from the original $f_i$.  From this point on, all we need to do is apply Corollary \ref{slc} paying attention to the remarks about divergence in Section \ref{sdirect}.
 \end{itemize}


\begin{thebibliography}{100}
\bibitem{BeHa} M. B. Bekka, P. de la Harpe and A. Valette, Kazhdan's Property (T), Cambridge University Press, Cambridge 2008.
\bibitem{Bekka} M. B. Bekka and M. Mayer, Ergodic theory and topological dynamics of group actions on homogeneous spaces, in London Mathematical Society Lecture note series 269, Cambridge University Press, Cambridge, 2000.
\bibitem{BEG} M. Bjorklund, M. Einsiedler, A. Gorodnik, Effective multiple mixing for semisimple groups, In preparation.
\bibitem{Bor} A. Borel, Linear Algebraic Groups, Springer-Verlag, 1991.
\bibitem{HM} R. E. Howe and C. C. Moore, 'Asymptotic properties of unitary representations', J. Funct. Anal. 32 (1979), no. 1, 72–96.
\bibitem{HT} R. Howe and Eng Che Tan, Non Abelian Harmonic Analysis, Springer Verlag, 1992.
\bibitem{Hum}J. E. Humphreys, Linear Algebraic Groups, GTM v. 21, Springer Verlag, 1981.
\bibitem{KaSp} A. Katok and R. Spatzier, First cohomology of Anosov actions of higher rank Abelian groups and applications to rigidity, Inst. Hautes Etudes Sci. Publ. Math. 79 (1994), 131-156.
\bibitem{Lang} S. Lang, Real Analysis, Addison - Wesley, 1983.
\bibitem{Led} F. Ledrappier, Un champ markovien peut \^{e}tre d'entropie nulle et m\'elangeant, C. R. Acad. Sci. Paris Ser. A-B 287 (1978), no. 7, A561–A563. 
\bibitem{Mar} G.A. Margulis, Discrete subgroups of semisimple Lie groups, Ergebnisse der Mathematik und ihrer Grenzgebiete, Springer-Verlag, 1991.
\bibitem{Moz} S. Mozes, Mixing of all orders of Lie groups actions, Inventiones mathematicae, Volume 107, Issue 1, pp 235-241 (1992).
\bibitem{Oh} H. Oh, Uniform pointwise bounds for matrix coefficients of unitary representations and applications to Kazhdan constants. Duke mathematical journal, 113(1), 133-192.
\bibitem{Pla} V. Platonov and A. Rapinchuk, Algebraic Groups and Number Theory, Academic Press, 1994.
\bibitem{Sch} K. Schmidt, Dynamical systems of algebraic origin, Birkh\"auser Verlag, Basel, 1995. 
\bibitem{Starkov} A. N. Starkov, Dynamical systems on Homogeneous spaces, Translations of Mathematical Monographs Vol. 190, American Mathematical Society, 2000.
%\bibitem{Stein} E. Stein, Singular Integrals and Differentiability Properties of Functions, Princeton University Press, Princeton, New Jersey, 1970.
\bibitem{Tai} M.H. Taibleson, Fourier Analysis on Local Fields, Mathematical Notes, Princeton University Press, Princeton, New Jersey 1975.
\bibitem{Wang} Z. J. Wang, Uniform pointwise bounds for matrix coefficients of unitary representations on semidirect products, preprint, 2012.
\bibitem{War} G. Warner, Harmonic Analysis on semisimple Lie groups I, Springer Verlag, Berlin 1972.
\end{thebibliography}
\end{document}